\documentclass{article}
\usepackage{amsmath,amssymb,amsthm,bbm}
\usepackage{mathrsfs}
\usepackage{bbm} % это для 1
\usepackage[usenames,dvipsnames]{color}
\usepackage[affil-it]{authblk}
\title{\textbf{Phase transition in the exit boundary problem for random walks on groups}\thanks{Supported by the RNF grant 14-11-00581.}}

\author{A.\,M.~VERSHIK%
\thanks{e-mail: avershik@gmail.com}}
\affil{\small
St.~Petersburg Department of Steklov Institute of Mathematics,\\
St.~Petersburg State University,\\
Institute for Information Transmission Problems, Moscow}

\author{A.\,V.~MALYUTIN%
\thanks{e-mail: malyutin@pdmi.ras.ru}}
\affil{\small St.~Petersburg Department of Steklov Institute of Mathematics}

\date{}

\newcounter{numsec}[section]

\newtheoremstyle{Mystyle}
     {\topsep}
     {\topsep}
     {\it}%         Body font
     {}%         Indent amount (empty = no indent, \parindent = para indent)
     {\bfseries}% Thm head font
     { }%        Punctuation after thm head
     { }%     Space after thm head (\newline = linebreak)
     {\thmnumber{#2.~}\thmname{#1}\thmnote{ #3}.}%         Thm head spec
\theoremstyle{Mystyle}
\newtheorem{theorem}[numsec]{Theorem}
\newtheorem{proposition}[numsec]{Proposition}
\newtheorem{lemma}[numsec]{Lemma}

\newcommand \N      {\mathbb N}
\newcommand \Z      {\mathbb Z}

\newcommand \R      {\mathbb R}

\newcommand \ssm    {\smallsetminus}

\renewcommand \ge{\geqslant}

\newcommand \Erg     {\operatorname{Erg}}
\newcommand \Inv     {\operatorname{Inv}}

\newcommand \Paths   {\operatorname{\mathscr{P}}}

\newcommand \hor     {h}
\newcommand \Damma   {\operatorname{D}}

\begin{document}

\maketitle

\rightline{To the memory of E.~B.~Dynkin}
\begin{abstract}
We describe the full exit boundary of random walks on homogeneous trees, in particular, on the free groups. This model exhibits a phase transition, namely, the family of Markov measures under study loses ergodicity as a parameter of the random walk changes.

The problem under consideration is a special case of the problem of describing the invariant (central) measures on branching graphs, which covers a number of problems in combinatorics, representation theory, probability and was fully stated in a series of recent papers by the first author \cite{V1,V2,V3}. On the other hand, in the context of the theory of Markov processes, close problems were discussed as early as 1960s by E.~B.~Dynkin.\end{abstract}

\section{Introduction}
In the late 1960s,  E.~B.~Dynkin  \cite{D69,D70,D71} developed the concept of exit (respectively, entrance) boundaries\footnote{In the terminology of those papers, the  spaces of exits (entrances).} for Markov chains with given cotransition (respectively, transition) probabilities. He started essentially from Martin's theory and its probabilistic interpretation suggested by J.~Doob, G.~Hunt, and others. In the book
\cite{DyYU}, a number of problems of this theory is presented. For instance, the Martin boundary, i.\,e., the set of minimal harmonic functions, is described for the natural Laplace operator on the free group. In the book~\cite{Gui}, numerous results on various boundaries of classical symmetric spaces are collected.

 The link between these problems and the problem of finding the invariant (central) measures on the path spaces of branching graphs (i.\,e., locally finite countable connected  $\Bbb N$-graded graphs, or Bratteli diagrams) has been repeatedly emphasized in a number of papers by the first author and S.~V.~Kerov. In the recent paper
\cite{V3},  it is clearly stated that  {\it the problem of finding the exit boundary, i.\,e., the set of all Markov measures with given cotransition probabilities, has the same scope
 as the problem of describing the set of measures with a given cocycle (in the case of central measures, the cocycle is identically equal to one) for a hyperfinite equivalence relation on a Cantor-like set.} As emphasized in \cite{V3}, the most transparent and geometric context for these problems is the theory of projective limits of simplices. On the other hand, the most important facts discovered recently (\cite{V3}) are related to the properties of the tail filtrations of Markov processes (or the path spaces of branching graphs), such as standardness and compactness in the intrinsic metric.

In this paper we consider a special problem, namely, that of finding the exit boundary for simple random walks on trees, or, equivalently, the problem of describing the central measures for branching graphs related to these random walks. In subsequent papers we will consider a wider class of groups and branching graphs.

Now we describe the construction of \emph{dynamic graphs} in terms of which we study the problem. Let $T$ be an arbitrary connected locally finite graph without multiple edges with a distinguished vertex~$v_0$. The
\emph{dynamic graph} $\Damma (T,v_0)$ is the  $\N$-graded graph whose $n$th level is a copy of the set of vertices of $T$ connected with the distinguished vertex $v_0$ by walks
of length~$n$ (or, which is the same, by paths of length at most~$n$ and of the same parity as~$n$); the unique vertex of the zero level will be denoted by~$\varnothing$. Two vertices in $\Damma (T,v_0)$ are adjacent if and only if they lie at adjacent levels and are connected by an edge in the graph $T$. If a connected graph $T$ has no odd cycles, then choosing the initial vertex turns it into an $\N$-graded graph for which the dynamic graph $\Damma (T,v_0)$ is defined in~\cite{VN} and called the  \emph{pascalization} of $T$.\footnote{The term is due to the fact that if $T$ is the chain (i.\,e.,  $T=\{n\in \Bbb Z\}$ with the grading
 $n\rightarrow |n|\in \Bbb N$), then the ``pascalization'' of $T$ is the Pascal graph (the infinite Pascal triangle). As shown in~\cite{VN}, the branching graph of the infinite-dimensional Brauer algebra is the pascalization of the Young graph; for another example of pascalization, see~\cite{GK}.}

If $G$ is a countable group with a fixed finite collection of generators $A=A^{-1}$, then by the dynamic graph $\Damma(G,A)$ of the pair $(G,A)$ we mean the dynamic graph $\Damma(T(G,A))$, where $T(G,A)$ is the Cayley graph of the pair $(G,A)$; this graph will be called the \emph{dynamic Cayley graph}.\footnote{Note that the choice of the initial vertex does not affect the structure of the dynamic graph, since the group $G$ act transitively on the set of vertices of $T(G,A)$.}

In this paper $T$ will always be the homogeneous tree $T_{q+1}$, $q>0$, of valency~$q+1$. In particular, $T_2$
is the Cayley graph of the group  $\Bbb Z$, and $T_{2k}$ is the Cayley graph of the free group with $k$ generators for the set of standard generators and their inverses. The graph  $\Damma (T)$ will be called the \emph{dynamic graph of the tree $T$}, or the \emph{pascalization of the tree~$T$}.

Now recall the general definition of central and ergodic measures on the path space of an arbitrary branching graph
$\Damma$ (see \cite{V3}). Consider the Cantor-like set $\Paths(\Damma)$ of all infinite paths in the graph~$\Damma$. It is compact in the weak topology. We will consider Borel probability measures on this space. A measure~$\mu$ on~$\Paths(\Damma)$ is called \emph{central} if it has the following property: for every vertex $v$ of
 $\Damma$, the conditional measure on the set of finite paths connecting $\varnothing$  with $v$ is uniform. One can easily see that every central measure is Markov with respect to the grading, i.\,e., it is the law of a Markov chain for which the numbers of levels are time moments varying from
  $0$ to $+\infty$.\footnote{In the theory of dynamical systems, Markov measures are measures of maximal entropy on  Markov compacta.} It is easy to verify that all central measures have the same system of \emph{cotransition probabilities}, i.\,e., the cotransition (i.\,e., inverse) probability of going from an arbitrary vertex $u$ of the $n$th level to a fixed vertex $v$ of the $(n+1)$th level is proportional to the number of paths leading from $\varnothing$ to~$u$.

The central measures form a convex compact set, which is a compact simplex (see \cite{V3}) in the compact space of all measures on the path space. A measure on the path space is called \emph{ergodic}
(or \emph{regular}) if it is an extreme point in this simplex. Another definition of ergodicity involves the {\it tail filtration}.
 The latter is the natural filtration of $\sigma$-algebras  $\{{\frak A}_n\}_{n \in \Bbb N}$, where ${\frak A}_n$ is the $\sigma$-algebra of sets described in terms of conditions on the coordinates (vertices) of paths with numbers $>n$. A central measure $\mu$ is erdodic if the intersection over $n$ of the sequence of these $\sigma$-algebras, regarded as $\sigma$-algebras of $\mu$-measurable sets, is the trivial $\sigma$-algebra consisting of sets of zero or full measure.

Now we state our main problem in application to the dynamic graph of a homogeneous tree. Of course, the statement of the problem makes sense for an arbitrary branching graph.

\textbf{Describe the set $\Erg(\Damma(T_{q+1}))$ of all ergodic central measures on the compact space $\Paths(\Damma(T_{q+1}))$ of infinite paths in the graph $\Damma(T_{q+1})$, or, in other terms, describe
 all ergodic Markov measures of maximal entropy on  $\Paths(\Damma(T_{q+1}))$.}

In this situation, we call  $\Erg(\Damma(T_{q+1})$ the \emph{exit boundary}. Since we deal with Markov measures,  to ``describe'' here means to find \emph{the transition probabilities of the corresponding Markov chains}. This is what will be done.

Observe the difference between this problem and that of finding the Poisson--Furstenberg boundary for simple random walks on groups. In the latter problem, one has a fixed Markov chain corresponding to a simple random walk (i.\,e., a measure on the path space of a dynamic graph) and must find the exit boundary of the {\it given Markov chain} (see \cite{KV}). In other words, one must represent this Markov measure as the integral of  ergodic central measures over some (harmonic) measure. The space of ergodic central measures equipped with this harmonic measure is exactly the Poisson (or Poisson--Furstenberg  \cite{V2000}) boundary. It is a measure space. But in our case we want to find all ergodic central Markov measures. The answer is a topological space without any distinguished measure. The difference is seen already for the group
 $\Bbb Z$: here the Poisson boundary is trivial, while the ``full'' exit boundary is an interval.

In this connection, it makes sense to consider this simplest special case of our problem, namely, the case of the tree
$T_2$ (chain), since we will need it in what follows. As mentioned above,  $\Damma (T_2)$ is the Pascal graph, and we arrive at the problem of describing the central measures on the Pascal graph. Note that the path space of the Pascal graph is the space of all  $0-1$ sequences, and a measure on the path space is central if and only if it is invariant under the group of finite permutations. Hence, applying an appropriate version of de Finetti's theorem, we see that an ergodic measure is a Bernoulli measure, with probability $p$ of $0$ and probability $1-p$ of $1$. Thus the exit boundary in this case is the interval
 $[0,1]$; its extreme points $\{0\}$ and $\{1\}$ correspond to the Dirac measures on the two extreme rays of the Pascal triangle. It is more proper to regard this boundary as the cone over the two extreme points. Note that here ergodicity means not Kolmogorov's  $0-1$ law, but the so-called Hewitt--Savage law. As noted above, the Poisson boundary in this example is trivial (consists of a single point), and the exit boundary is the interval $[0,1]$. In this sense, the main result of the paper, a description of the exit boundary for an arbitrary homogeneous tree, can be regarded as another (``noncommutative'') generalization of de Finetti's theorem. The exit boundary problem for the dynamic Cayley graph of an arbitrary pair $(G,A)$ where $G$ is a group and $A$
is a set of its generators and their inverses (see above) can be regarded as an analog of de Finetti's problem on symmetric measures. Apparently, in all group-theoretic examples the tail filtration is standard in the sense of
 \cite{V3} and the compactness with respect to the intrinsic metric holds;
 in the case of the graph~$\Damma (T)$, where $T$ is a tree, considered here, this follows from the main theorem on the exit boundary.

The most important consequence of our analysis is the existence of a phase transition; it can be described in two equivalent ways: 1) the phase transition is due to the fact that the central measure loses ergodicity as the drift rate changes, 2) the phase transition is due to the fact that an eigenfunction of the Laplace operator loses minimality under a deformation  preserving the eigenvalue. Note that the theory of eigenfunctions of the Laplace operator on trees, on Riemannian surfaces of constant negative curvature, on hyperbolic groups is the subject of many papers (those closest to our topic are \cite{F-TN, Helg}). In these cases, one usually proves the Poisson formula which allows one to represent eigenfunctions by integrals over minimal eigenfunctions. However, we are interested not only in the linear theory of eigenfunctions, but also in their relation to the order: it is the consideration of positive eigenfunctions and their interpretation as Markov (central) measures that reveals the phase transition phenomenon. A question of great interest is for which graphs, groups (including Lie groups), or homogeneous spaces this effect takes place. One may conjecture that this is the case for some symmetric spaces of semisimple Lie groups.

\section{Statement of the main result and some comments}
\label{sec:1}

\paragraph{Definition. The family of measures $\lambda_{\omega,r}$.}
Let $q\ge 2$ be a positive integer, $T_{q+1}$ be the $(q+1)$-homogeneous tree, and $\Damma(T_{q+1})$ be the dynamic graph (see the definition above). The natural projection $\Damma(T_{q+1})\to T_{q+1}$ will be denoted by $\pi$. The space of ends of the tree $T_{q+1}$ will be denoted by $\partial T_{q+1}$; for $q\ge 2$, this is a Cantor-like set. If
$v$ and $w$ are adjacent vertices in $\Damma(T_{q+1})$, with $v$ belonging to a higher level  than $w$, we say that \emph{$w$ immediately succeeds $v$} and write  $v\prec w$.
Let $\omega\in\partial T_{q+1}$. An edge $(v,w)$  (with $v\prec w$) of the graph $\Damma(T_{q+1})$ is called \emph{$\omega$-directed} if the projection $\pi(w)$ lies between\footnote{Every vertex of a tree is connected by a unique geodesic path with every end $\omega$, hence for every vertex~$x\in T_{q+1}$ and every end~$\omega\in\partial T_{q+1}$ there is a unique vertex adjacent to $x$ that lies between $x$ and~$\omega$.} $\pi(v)$ and~$\omega$.

Given an end $\omega\in\partial T_{q+1}$ and a number $r\in[0,1]$, denote by~$\lambda_{\omega,r}$ the Markov measure on the space $\Paths(\Damma(T_{q+1}))$ of all infinite paths in  $\Damma(T_{q+1})$ for which the transition probabilities coincide and are equal to $r$ on all $\omega$-directed edges, and also coincide (and are equal, correspondingly, to
 $\frac{1-r}{q}$) on all the other edges.

Obviously, all measures~$\lambda_{\omega,r}$ are central.

\begin{theorem}
\label{thm:exit}
For $q\ge2$, the set $\Erg (\Damma(T_{q+1}))$ of all ergodic central measures on the space $\Paths(\Damma(T_{q+1}))$ of infinite paths in the dynamic graph $\Damma(T_{q+1})$ over the $(q+1)$-homogenelous tree~$T_{q+1}$ (i.\,e., the exit boundary) coincides with the following family of Markov measures:
\begin{equation*}
\label{eq:family}
\Lambda_q:=\left\{\lambda_{\omega,r} \mid \omega\in\partial T_{q+1}, r\in\left[1/2,1\right]\right\}.
\end{equation*}
Thus the exit boundary is homeomorphic (in the weak topology) to the product
$$
\partial T_{q+1} \times \left[1/2,1\right].
$$
\end{theorem}

\paragraph{Remarks.}
1. Homogeneity. The theorem implies that all ergodic central Markov measures on $\Paths(\Damma(T_{q+1}))$ are time-homogeneous (in the sense that the corresponding transition probabilities depend only on a vertex of the tree, but not on the moment when this vertex is being visited), since they belong to the described family, whose all members have this property. For arbitrary central measures, homogeneity does not hold in general.

2. The parameter $r$ has a simple expression in terms of the rate $\tau$ of approaching the end  $\omega\in\partial T_{q+1}$: $\tau=2r-1$. In this case $|\tau|$ coincides with the limit of the ratio $\frac{d(\pi(\varnothing),\pi(w_n))}{n}$ (which is the drift rate, i.\,e., the velocity with which $\pi(w_n)$ goes to infinity). For
$r>\frac{1}{2}$, the number $\tau$ can also be described as the \emph{stabilization rate},  i.\,e., the limit of the ratio $\frac{s(n)}{n}$, where $s(n)$ is the greatest number of a vertex on the ray  $[\pi(\varnothing),\omega\rangle$ among all the vertices of this ray that will not occur again on the given random walk (see Lemma~\ref{lem:limitpoint}).
The value $r=1$ corresponds to the Markov measure that is the Dirac measure on the geodesic, the deterministic path leading from the initial point to the boundary point~$\omega$.

3. According to the theorem, the Markov measures corresponding to $r \in [0,1/2)$ are central, but no longer ergodic; thus at
 $\frac12$ we have a {\it phase transition, loss of ergodicity type}, see Section~\ref{sec:PHT} below. The decomposition of these nonergodic central measures into ergodic components is given in the same section.

4. The Poisson--Furstenberg boundary is contained in the family  $\Lambda_q$ as the subset $$\{\lambda_{\omega,\frac{q}{q+1}} : ~\omega\in\partial T_{q+1} \}.$$

5. It is of interest to consider adic transformations of paths (similar to the Pascal automorphism) and the tail filtrations corresponding to the Markov measures $\lambda_{\omega,r}$. It follows from our results that these filtrations are standard in the sense of \cite{V3}.

\section{The scheme of the proof and the first lemmas}

The proof of our main result, Theorem~\ref{thm:exit}, can obviously be reduced to proving the following two assertions.
\begin{itemize}
\item[I.] Every ergodic central measure on the  space $\Paths(\Damma(T_{q+1}))$ has the form $\lambda_{\omega,r}$ (Proposition~\ref{pro:main-2}).
\item[II.] A measure of the form $\lambda_{\omega,r}$ is ergodic if and only if $r\in[1/2,1]$ (Proposition~\ref{pro:main-3}).
\end{itemize}

The proof of the first one essentially relies on the almost everywhere convergence of martingales (in other words, the pointwise ergodic theorem) applied to the tail filtration on the path space of the graph.
\begin{lemma}
\label{lem:ergo-0}
Let $\nu$ be an ergodic central measure on the space $\Paths(\Damma)$ of infinite paths of a dynamic graph
$\Damma$, and $p_\nu$ be the corresponding system of transition probabilities on the edges of~$\Damma$. Then

{\rm(i)} For every finite path $R=(\varnothing=v_0, \dots, v_k)$ in $\Damma$ and for $\nu$-almost every ($\nu$-a.\,e.)\  infinite path $(w_i)_{i\in\N_0}$ in~$\Damma$, the sequence\footnote{We use the standard notation of the theory of branching graphs: the number of paths leading from a vertex $x$ to a vertex $y$ is denoted by $\dim(x,y)$ (and $\dim(\varnothing,y)$ is denoted by $\dim(y)$). It is motivated by  the algebraic interpretation of Bratteli diagrams.}
 $\dim(v_k,w_i)/\dim(w_i)$, $i\in\N_0$, has a limit, and
$$
\lim_{i\to\infty}\frac{\dim(v_k,w_i)}{\dim(w_i)}=p_\nu(v_0,v_1)\cdots p_\nu(v_{k-1},v_k).
$$

{\rm (ii)} If $(v,x)$ is an edge in $\Damma$, then for $\nu$-a.\,e.\ path $(w_i)_{i\in\N_0}$ in~$\Damma$ the sequence $\dim(x,w_i)/\dim(v,w_i)$, $i\in\N_0$, has a limit, and
\begin{equation}
\label{eq:egro-2}
\lim_{i\to\infty}\frac{\dim(x,w_i)}{\dim(v,w_i)}=p_\nu(v,x).
\end{equation}
\end{lemma}

\paragraph{Definition.} A path $(w_i)_{i\in\N_0}$ in the graph $\Paths(\Damma(T_{q+1}))$ will be called \emph{typical} with respect to a given ergodic measure $\nu$ if the conditions of Lemma~\ref{lem:ergo-0} hold for this path for every finite path $R$ and every edge of the graph.
Clearly, the set of typical paths is of full $\nu$-measure.

\paragraph{Definition.}
We say that a sequence $(x_i)_{i\in\N}$ in the tree $T_{q+1}$ \emph{converges} to a point~$\omega\in\partial T_{q+1}$ if for some (and hence every) vertex $v\in T_{q+1}$ the length of the common part
$
[v, x_i]\cap [v,\omega\rangle
$
of the geodesic segment (path) $[v,x_i]$ from $v$ to $x_i$ and the ray $[v,\omega\rangle$ from $v$ to $\omega$ tends to infinity as $i$ grows.

If for an infinite path in~$\Damma(T_{q+1})$ its projection in $T_{q+1}$ converges to a point $\omega\in\partial T_{q+1}$,
we say that the path itself \emph{converges to~$\omega$}.

\begin{lemma}
\label{lem:limitpoint}
If a central measure $\nu$ on the path space~$\Paths(\Damma(T_{q+1}))$ is ergodic, then there exists a point in $\partial T_{q+1}$ to which $\nu$-a.\,e.~path converges.
\end{lemma}

\begin{proof}
It is easy to check that the random walk in $T_{q+1}$ corresponding to the measure $\nu$ is nonrecurrent: for any vertex
$v\in T_{q+1}$, almost every (with respect to~$\nu$) path visits the set $\pi^{-1}(v)$ at most finitely many times. Therefore, since $T_{q+1}$ is a tree, $\nu$-a.\,e.\ path converges to a point in~$\partial T_{q+1}$. It remains to observe that, by ergodicity, the limiting points of almost all paths must coincide.\footnote{Here we have not assumed the time-homogeneity; this property for ergodic measures will be proved in the next section.}
\end{proof}

\section{Homogeneity}

This section is devoted to the proof of the following Proposition~\ref{pro:main-2}, which constitutes the first part of Theorem~\ref{thm:exit}.

\begin{proposition}
\label{pro:main-2}
Every ergodic central measure on the space $\Paths(\Damma(T_{q+1}))$ of infinite paths in the dynamic graph~$\Damma(T_{q+1})$ has the form~$\lambda_{\omega,r}$ for some $\omega \in \partial T_{q+1}$, $r\in [0,1]$, i.\,e., for every ergodic central measure~$\nu$ there is a point $\omega\in \partial T_{q+1}$ such that the transition probabilities corresponding to~$\nu$ coincide on all  $\omega$-directed edges, and also coincide on all the other edges.
\end{proposition}

\begin{proof}[Proof.]
First consider the case of a \emph{nondegenerate}\footnote{A measure $\nu$ on $\Paths(\Damma(T_{q+1}))$ is called \emph{nondegenerate} if for every finite path~$P$ in $\Damma(T_{q+1})$ the measure $\nu(P)$ of the set of infinite paths starting with~$P$ is positive. (In other words, a measure is nondegenerate if its transition probabilities do not vanish on any edge.)} ergodic measure~$\nu$.

By Lemma~\ref{lem:limitpoint},
the measure~$\nu$ determines a point $\omega\in\partial T_{q+1}$ to which $\nu$-a.\,e.\ path converges. We will show that
$\nu=\lambda_{\omega,r}$ for some $r\in[0,1]$.

First note that at every vertex $v\in\Damma(T_{q+1})$, the transition probability $p_\nu$ takes equal values on  those edges outgoing from $v$ that are not $\omega$-directed. Indeed, if
 $(v,v')$ and $(v,v'')$ are two edges that are not $\omega$-directed and $(w_i)_{i\in\N_0}$ is a $\nu$-typical path converging to~$\omega$, then, by symmetry considerations, for all sufficiently large $j\in\N$ we have $\dim(v',w_j)=\dim(v'',w_j)$, which, since $(w_i)_{i\in\N_0}$ is a typical path, by Lemma~\ref{lem:ergo-0} implies that
\begin{equation*}
p_\nu(v,v')=\lim_{j\to\infty}\frac{\dim(v',w_j)}{\dim(v,w_j)}=\lim_{j\to\infty}\frac{\dim(v'',w_j)}{\dim(v,w_j)}=p_\nu(v,v'').
\end{equation*}

Since every vertex of the dynamic graph $\Damma(T_{q+1})$ has exactly one outgoing $\omega$-directed edge,
the fact that $\nu$ belongs to the family $\{\lambda_{\omega,r}:r\in[0,1]\}$ is implied by the following assertion.

\emph{The transition probabilities $p_\nu$ corresponding to the measure $\nu$ coincide on all
$\omega$-directed edges.}

Since the graph $\Damma(T_{q+1})$ is connected, it suffices to check only that
 \emph{the transition probabilities $p_\nu$ coincide on $\omega$-directed edges outgoing from adjacent vertices.}

Let $a\prec b$ be adjacent vertices in~$\Damma(T_{q+1})$.
Given a vertex~$v$, by $v_\omega$ we denote the end vertex of the  $\omega$-directed edge outgoing from~$v$.
We must show that $p_\nu(a,a_\omega)=p_\nu(b,b_\omega)$.

As above, let $(w_i)_{i\in\N_0}$ be a typical path converging to $\omega$.
Then, by Lemma~\ref{lem:ergo-0}, we have
\begin{equation}
\label{eq:egro-2=2}
\begin{split}
\lim_{i\to\infty}\frac{\dim(a_\omega,w_i)}{\dim(a,w_i)}&=p_\nu(a,a_\omega),\\
\lim_{i\to\infty}\frac{\dim(b_\omega,w_i)}{\dim(b,w_i)}&=p_\nu(b,b_\omega).
\end{split}
\end{equation}
We will show that there are infinitely many $j\in\N$ such that
\begin{equation}
\label{eq:ii+1}
\begin{split}
\dim(a,w_j)&=\dim(b,w_{j+1}),\\
\dim(a_\omega,w_j)&=\dim(b_\omega,w_{j+1}),
\end{split}
\end{equation}
which, in view of~\eqref{eq:egro-2=2}, immediately implies the desired equation $p_\nu(a,a_\omega)=p_\nu(b,b_\omega)$.

Note that $\dim(x,y)$, the number of paths leading from a vertex $x$ to a vertex $y$ in the graph $\Damma(T_{q+1})$, coincides with the number of walks of length $|L(x)-L(y)|$ between the vertices $\pi(x)$ and $\pi(y)$ in the tree~$T_{q+1}$ (hereafter by $L(z)$ we denote the number of the level containing a vertex $z$). Besides, note that in a homogeneous tree the number of walks of given length between two vertices depends only on the distance between these vertices. Finally, observe that, since
 $L(b)=L(a)+1$, $L(v_\omega)=L(v)+1$, $L(w_{i+1})=L(w_i)+1$  by construction, for every $i\in\N$ we have (cf.~\eqref{eq:ii+1})
\begin{equation}
\label{eq:ii+1-1}
\begin{split}
L(a)-L(w_i)&=L(b)-L(w_{i+1}),\\
L(a_\omega)-L(w_i)&=L(b_\omega)-L(w_{i+1}).
\end{split}
\end{equation}
It follows that~\eqref{eq:ii+1} holds for every $j\in\N$ for which
\begin{equation}
\label{eq:ii+1-2}
\begin{split}
d(\pi(a),\pi(w_j))&=d(\pi(b),\pi(w_{j+1})),\\
d(\pi(a_\omega),\pi(w_j))&=d(\pi(b_\omega),\pi(w_{j+1})).
\end{split}
\end{equation}

Thus it remains to prove that there are infinitely many $j\in\N$ for which~\eqref{eq:ii+1-2} holds.

Two cases are possible:
\begin{itemize}
\item[(i)] $b=a_\omega$, i.\,e., the edge $(a,b)$ is $\omega$-directed,
\item[(ii)] $b\neq a_\omega$, i.\,e., the edge $(a,b)$ is not $\omega$-directed.
\end{itemize}

Since the walk $(\pi(w_i))_{i\in\N_0}$ moves along adjacent vertices of the graph $T_{q+1}$, for any $u\in T_{q+1}$ and $i\in\N$ we have
$$
d(u,\pi(w_{i+1}))=d(u,\pi(w_i))\pm1.
$$
Since $(\pi(w_i))_{i\in\N_0}$ converges to $\omega$, in case (i) relations~\eqref{eq:ii+1-2} hold whenever $j$ is sufficiently large and
\begin{equation}
\label{eq:udal}
d(\pi(a),\pi(w_{j+1}))=d(\pi(a),\pi(w_{j}))+1,
\end{equation}
and in case (ii) relations~\eqref{eq:ii+1-2} hold whenever $j$ is sufficiently large and
\begin{equation}
\label{eq:pribl}
d(\pi(a),\pi(w_{j+1}))=d(\pi(a),\pi(w_{j}))-1.
\end{equation}

Moments for which \eqref{eq:udal} holds are infinitely many, since $(\pi(w_i))_{i\in\N_0}$ moves along adjacent vertices and converges to~$\omega$. Moments for which~\eqref{eq:pribl} holds are also infinitely many, since otherwise
all edges $(w_i,w_{i+1})$ starting from some $i$ would be $\omega$-directed, which contradicts (as follows from Lemma~\ref{lem:ergo-0}) the assumption that $\nu$ is nondegenerate.
\medskip

\emph{The case of a degenerate measure.}

In the case of a degenerate ergodic measure~$\nu$, when in~$\Damma(T_{q+1})$ there are $\nu$-inaccessible vertices and the transition probability  $p_\nu$ vanishes on some edges, we will prove that $\nu$ has the form~$\lambda_{\omega,1}$.

Indeed, if a vertex $v$ is $\nu$-inaccessible, i.\,e., if the measure of the set of infinite paths passing through $v$ vanishes, then, by the centrality of $\nu$, all vertices succeeding $v$ (i.\,e., vertices lying at lower levels and connected with $v$ by paths) are also $\nu$-inaccessible. Besides, recall that, by Lemma~\ref{lem:limitpoint}, the ergodic measure~$\nu$ determines a point $\omega\in\partial T_{q+1}$ to which $\nu$-a.\,e.\ path converges. But the set of paths converging to $\omega$ that pass neither through $v$ nor through vertices succeeding $v$ is at most countable. Thus $\nu$ is supported  by a set of paths that is at most countable. By the ergodicity of $\nu$ it follows that it is supported by a single path, and the centrality implies that this path forms a one-point class of the tail partition. Such a path for $\omega$ is unique, and the measure supported by this path is~$\lambda_{\omega,1}$.
\end{proof}

\section{The list of ergodic central measures, and the decomposition of nonergodic measures into ergodic components}

To complete the proof of Theorem~\ref{thm:exit}, it remains to prove the following proposition.

\begin{proposition}
\label{pro:main-3}
A measure of the form $\lambda_{\omega,r}$ is ergodic if and only if $r\in[1/2,1]$.
\end{proposition}

\begin{proof}[Proof of the nonergodicity for $r<1/2$]\
Let $\omega\in\partial T_{q+1}$ and $r\in[0,1/2)$.
Consider the measure $\lambda_{\omega,r}$ and the corresponding Markov process on~$T_{q+1}$. This process is time-homogeneous by the definition of  $\lambda_{\omega,r}$; its transition probability  on a directed edge  $(x,y)$ is equal to $r$ if $y$ lies between $x$ and $\omega$, and  $\frac{1-r}{q}$ otherwise.

We see that a.\,e.\ trajectory of the process converges to a point of $\partial T_{q+1}\ssm\{\omega\}$, since it moves away from $\omega$ along $\omega$-horospheres with drift
$1-2r$ and moves along adjacent vertices. It follows that the distribution on $\partial T_{q+1}\ssm\{\omega\}$ of the limits of trajectories of this Markov process is given by a continuous measure, since the Markov process itself, and hence this limiting measure on  $\partial T_{q+1}\ssm\{\omega\}$, are invariant under automorphisms of the tree $T_{q+1}$ leaving the ray $[\pi(\varnothing),\omega\rangle$ unchanged, and the group of such automorphisms has no finite orbits in
 $\partial T_{q+1}\ssm\{\omega\}$.

This continuity of the distribution of limiting points in ${\partial T_{q+1}}$ shows that $\lambda_{\omega,r}$ is not ergodic, since Lemma~\ref{lem:limitpoint} implies that almost all trajectories of the Markov process in $T_{q+1}$ corresponding to an ergodic central measure converge to the same point in~${\partial T_{q+1}}$.

\smallskip
\noindent{\it Proof of the ergodicity for $r\ge 1/2$.}\footnote{To prove that the measures $\lambda_{\omega,r}$, $r\ge 1/2$, are ergodic, we could use the entropy criterion from \cite{KV}, which applies to our situation, but we use instead the fact that a list of measures containing all ergodic measures is already found, and it remains to pick those of them that are not integral combinations of other measures. Note that the standardness of the filtration, whenever it holds, gives a powerful ergodicity criterion.} It suffices to observe the following. First, we prove that almost every path with respect to the measure $\lambda_{\omega,r}$ for $r\ge \frac{1}{2}$ tends to $\omega$ with rate $2r-1$ (see Lemma~\ref{lem:lambdawr} below). On the other hand, every nonergodic measure can be uniquely decomposed into an integral over ergodic measures, which, by the above, are contained among the members of the family
$\Lambda_q=\left\{\lambda_{\omega,r} \mid \omega\in\partial T_{q+1}, r\in\left[1/2,1\right]\right\}$.
But a measure $\lambda_{\omega_0,r_0}\in\Lambda_q$ cannot be written as an integral over a measure on~$\Lambda_q$ that is not the Dirac measure at  $(\omega_0,r_0)$,  but is supported by a set of parameters ${\omega,r}$, since in this case the limits of $\lambda_{\omega_0,r_0}$-almost all  paths and their convergence rates would be different from the values $\omega_0$ and/or $2r_0-1$, respectively, with which they must coincide for almost all paths.
\end{proof}

\begin{lemma}
\label{lem:lambdawr}
Assume that $\omega\in\partial T_{q+1}$ and $r\in[1/2,1]$.
In the tree $T_{q+1}$, denote by $v_k$ the vertex of the ray $[\pi(\varnothing),\omega\rangle$ at distance
$k\in\N_0$ from the initial vertex $\pi(\varnothing)$. Then $\lambda_{\omega,r}$-a.\,e.\
 path $(w_i)_{i\in\N_0}$ converges to $\omega$ and
\begin{equation*}
\label{eq:LLN-1}
\frac{d(\pi(w_n),v_{\lfloor (2r-1) n \rfloor})}{n}\xrightarrow{n\to\infty} 0.
\end{equation*}
\end{lemma}

\begin{proof}
It is convenient to base the proof on the projection
\begin{equation}
\label{eq:Z-RW-0}
\Damma(T_{q+1})\xrightarrow{\pi} T_{q+1} \xrightarrow{-\hor_{\omega,v_0}} \Z,
\end{equation}
where $\hor_{\omega,v_0}$ is the horofunction\footnote{A horofunction on a countable metric space
$(X,d)$ is a function that is unbounded from below and is the pointwise limit of functions of the form $d(x,x_0)+C$ where $x_0\in X$ and $C\in\R$.} on $T_{q+1}$ tending to~$-\infty$ on the rays representing~$\omega$ and vanishing  at the point $v_0=\pi(\varnothing)$. This projection sends the measure $\lambda_{\omega,r}$ to the random walk on~$\Z$ with transition probabilities
\begin{equation}
\label{eq:Z-RW-1}
p(z,z+1)=r, \qquad p(z,z-1)=1-r.
\end{equation}
The  key observation is as follows. Whenever the projection $(-\hor_{\omega,v_0}(\pi(w_n)))_{n\in\N_0}$ renews its maximum in~$\Z$, the point $\pi(w_n)$ hits the ray $[\pi(\varnothing),\omega\rangle$. We omit the details.
\end{proof}

Let us give an explicit formula for the decomposition into ergodic components.

\paragraph{The decomposition of a nonergodic measure $\lambda_{\omega,r}$.}
For $r<1/2$, the central measure $\lambda_{\omega_0,r}$ decomposes into ergodic components $\lambda_{\omega,1-r}$:
\begin{equation*}
\label{eq:Radon-Nik-0}
\lambda_{\omega_0,r} = \int\limits_{\partial T_{q+1}} \lambda_{\omega,1-r}\cdot\rho_{\omega_0,r}(d\omega).
\end{equation*}
The distribution $\rho_{\omega_0,r}$ is absolutely continuous with respect to the harmonic measure
$\theta=\theta_{v_0}$ on $\partial T_{q+1}$ symmetric with respect to the initial vertex $v_0=\pi(\varnothing)$. The Radon--Nikodym derivative $\frac{d\rho_{\omega_0,r}}{d\theta}$ is given by the formula
\begin{equation}
\label{eq:Radon-Nik-1}
\frac{d\rho_{\omega_0,r}}{d\theta}(\omega)=\frac{(1-2r)(q+1)}{q-qr-r}\left(\frac{qr}{1-r}\right)^{\ell(\omega_0,\omega)},
\end{equation}
where $\ell(\omega,\omega_0)$ is the length of the common part of the rays $[v_0,\omega\rangle$ and $[v_0,\omega_0\rangle$, and the expression $0^0$ appearing for $r=\ell(\omega_0,\omega)=0$ should be interpreted as~$1$.

To derive formula~\eqref{eq:Radon-Nik-1}, we use the fact that the above projection~\eqref{eq:Z-RW-0} sends the measure $\lambda_{\omega,r}$ to the random walk on~$\Z$ with the transition probabilities~\eqref{eq:Z-RW-1},
and for $r<1/2$ the fraction of trajectories of this walk that reach the point $k\in\N_0\subset\Z$ but not $k+1$ is equal to
$$
\left(\frac{r}{1-r}\right)^k-\left(\frac{r}{1-r}\right)^{k+1}.
$$

\section{Phase transition, the Martin boundary}
\label{sec:PHT}

In this section we discuss an interpretation of the above-mentioned phase transition, as well as the properties of the Martin boundary of the graph~$\Damma(T_{q+1})$.

As we have seen, as the parameter $r$ passes through the value~$\frac{1}{2}$, the measures~$\lambda_{\omega,r}$
lose ergodicity. For the critical value, the ergodicity still holds, but the behavior of trajectories changes. If
 $r>\frac{1}{2}$, then $\lambda_{\omega,r}$-almost all paths tend to the limit~$\omega$ with the linear (in~$n$) rate $2r-1$. If $r=\frac{1}{2}$, then almost all paths still tend to~$\omega$, but with a sublinear rate; in this case, the behavior of paths is worth a more detailed study, which we do not present here.

To describe the main interpretation of the phase transition, we establish a correspondence between the family of measures
$\lambda_{\omega,r}$ and a special family of eigenfunctions of the Laplace operator on the homogeneous tree~$T_{q+1}$.

First recall that the set $\mathscr{H}_{\min}$ of minimal positive harmonic (i.\,e., invariant under the Laplace operator) functions  on the homogeneous tree~$T_{q+1}$ coincides with the family of functions of the form $q^{-\hor(v)}$ where $\hor(v)$ is an arbitrary horofunction on~$T_{q+1}$. It is easy to check that for every (real or complex) $\alpha$, the power $(q^{-\hor(v)})^\alpha=q^{-\alpha\hor(v)}$ of the minimal harmonic function~$q^{-\hor(v)}$ is an eigenfunction of the Laplace operator with the eigenvalue
\begin{equation*}
%\label{eq:s-alpha}
s_\alpha=\frac{q^\alpha+q^{1-\alpha}}{q+1}.
\end{equation*}

On the other hand, every positive eigenfunction of the Laplace operator on the base graph determines a central Markov measure on the path space of the corresponding dynamic graph: the transition probability
 $p_{\mu}$ of the central measure $\mu=\mu_f$ corresponding to an eigenfunction $f$ with eigenvalue~$s$ on an edge $(v,w)$ with $v\prec w$ is given by the formula
\begin{equation}
\label{eq:p_mu}
p_{\mu}(v,w)=\frac{f(\pi(w))}{f(\pi(v))(q+1)s}.
\end{equation}
Formula~\eqref{eq:p_mu} provides a bijection between the normalized positive eigenfunctions of the Laplace operator on $T_{q+1}$ and the nondegenerate central measures that, being Markov measures on the path space of $\Damma(T_{q+1})$,
are time-homogeneous. Under this bijection, minimal eigenfunctions correspond to ergodic central measures.

If $\omega$ is an end of the tree~$T_{q+1}$,  $\hor_{\omega}$ is a horofunction on $T_{q+1}$ tending to~$-\infty$ on the rays representing~$\omega$, and $\alpha\in\R$, then the map
 \eqref{eq:p_mu} sends the positive eigenfunction $q^{-\alpha\hor_{\omega}(v)}$ to the central measure
$\lambda_{\omega,r_\alpha}$ where\footnote{For $\alpha=\frac12$, we obtain $r_{\frac12}=\frac12$.}
\begin{equation}
\label{eq:r-alpha}
r_\alpha=\frac{1}{1+q^{1-2\alpha}}.
\end{equation}

Thus for every fixed~$\omega_0$ from $\partial T_{q+1}$ we have established a correspondence $r\leftrightarrow\alpha$ between the measures $\lambda_{\omega_0,r}$, more exactly, the parameters~$r\in(0,1)$, and the eigenfunctions $q^{-\alpha\hor_{\omega_0}(v)}$, i.\,e., the parameters~$\alpha\in\R$.
Since formula~\eqref{eq:p_mu} gives a bijection between the minimal (normalized positive) eigenfunctions and the ergodic (central Markov) measures, it follows from Proposition~\ref{pro:main-3} that the eigenfunction $q^{-\alpha\hor_{\omega}(v)}$ is minimal if and only if $\alpha\in[1/2,+\infty)$.

Now we can describe our phase transition not only in terms of ergodicity, but also in terms of eigenfunctions. As the exponent $\alpha$ passes through $\frac12$, the power $q^{-\alpha\hor(v)}=(q^{-\hor(v)})^\alpha$ of the harmonic function $q^{-\hor(v)}$ remains an eigenfunction of the Laplace operator, but ceases to be minimal, since only minimal measures correspond to ergodic measures. This interpretation is closer to standard models of phase transitions, since the parameter~$\alpha$   now is merely an exponent. The authors do not know whether such natural algebraic models of phase transition appeared earlier.

\medskip

In conclusion we would like to discuss another corollary concerning the Martin boundary, which follows from the properties of our example. We have mentioned the interpretation of the classical construction of the Martin boundary in terms of  branching graphs (Bratteli diagrams), see~\cite{Ker96, KOO, V3}. In~\cite{V3}  it is explained how the Martin boundary can be defined geometrically in the framework of the theory of projective limits of finite-dimensional simplices. The Martin boundary is embedded into the simplex of central Markov measures, and its \emph{minimal part} is identified with the set of extreme points of the simplex (the exit boundary).

In the case of the graph $\Damma(T_{q+1})$, its Martin boundary $\mathcal{M}=\mathcal{M}(\Damma(T_{q+1}))$ contains the exit boundary $\mathcal{M}_0=\Erg(\Damma(T_{q+1}))$ as a proper subset: the complementary set
 $\mathcal{M}'=\mathcal{M}\ssm \mathcal{M}_0$ is formed by the countable set of isolated points whose elements are in a one-to-one correspondence with the vertices of the tree~$T_{q+1}$.
This correspondence $\sigma\colon T_{q+1}\to \mathcal{M}'$ can be extended to a continuous embedding $T_{q+1}\cup\partial T_{q+1}\to \mathcal{M}$ which sends the boundary $\partial T_{q+1}$ to  $\partial T_{q+1}\times\{1/2\}$. The point $\sigma(v)\in \mathcal{M}'$ corresponding to a vertex $v\in T_{q+1}$  is a limiting point for the set $\pi^{-1}(v)$ in $\Damma(T_{q+1})$.

Note that finding the Martin boundary involves direct calculations with powers of convolutions or enumeration of walks of given length between vertices. The number of walks of length $n$ connecting two vertices at distance~$k\in\N_0$ in the $(q+1)$-homogeneous tree~$T_{q+1}$ for $n\ge k$ and even $n-k$ is given by the formula (see also \cite{LM71, Gri77, Pag93})
$$
S_q(n,k) ~=~ S_q(n,k+2) + q^{(n-k)/2}L(n,k) ~= \sum_{t\in\{k,k+2,\dots,n\}}q^{(n-t)/2}L(n,t),
$$
где
$$
L(n,t)~=~C_n^{(n-t)/2}-C_n^{(n-t)/2-1}
%&=\frac{2(t+1)}{n+t+2}C_n^{(n-t)/2}=\frac{2(t+1)}{n+t}C_n^{(n-t)/2-1}\\
~=~\frac{t+1}{n+1}C_{n+1}^{(n-t)/2}.
$$

The example with the Martin boundary of the graph $\Damma(T_{q+1})$ gives an answer (as expected, the negative one) to the question posed in~\cite{V3}: can one describe the Martin boundary in terms of the limiting simplex itself? Indeed, if
$$
\Sigma_0 \xleftarrow{\pi_{1,0}} \Sigma_1 \xleftarrow{\pi_{2,1}}\Sigma_2 \xleftarrow{} \dots \xleftarrow{} \Sigma_{n} \xleftarrow{\pi_{n+1,n}}\Sigma_{n+1}\xleftarrow{} \dots \Sigma_{\infty}=\Inv(\Damma(T_{q+1}))
$$
is the projective limit corresponding to the graph $\Damma(T_{q+1})$, then for the rarified approximation
$$
\Sigma_0 \xleftarrow{\pi_{1,0}\circ \pi_{2,1}} \Sigma_2 \xleftarrow{\pi_{3,2}\circ \pi_{4,3}}\Sigma_4 \xleftarrow {} \dots \Sigma_{\infty}=\Inv(\Damma(T_{q+1})),
$$
which has the same limiting simplex, the Martin boundary no longer contains points of the set $\mathcal{M}'$ that correspond to points of $T_{q+1}$ lying at an odd distance from $v_0$, since all simplices containing appropriate vertices are absent in the rarified sequence.

Thus the notion of the Martin boundary in the problem of describing the invariant measures on a given equivalence relation depends on an approximation of this relation and in this sense is not intrinsic for the problem.

\smallskip
Translated by N.~V.~Tsilevich.

\end{document}